\documentclass[12pt,leqno,fleqn]{amsart}  
\usepackage{amsmath,amstext,amsthm,amssymb,amsxtra}
\usepackage{graphicx}
\usepackage{txfonts} 
\usepackage[T1]{fontenc}
\usepackage{lmodern}

\usepackage{euler}   
  
\usepackage{mathtools}
\mathtoolsset{showonlyrefs,showmanualtags}

\usepackage{hyperref} 
\usepackage{amsrefs}  

\allowdisplaybreaks
\swapnumbers

\theoremstyle{plain} 
 
\newtheorem{proposition}[equation]{Proposition} 
\newtheorem{theorem}[equation]{Theorem} 

\newtheorem{cet}[equation]{Carleson Embedding Inequality}
\newtheorem{wcet}[equation]{Weighted Carleson Embedding Inequality}
\newtheorem{sntv}[equation]{Embedding Inequality of Sawyer and Nazarov-Treil-Volberg}

\theoremstyle{definition}

\theoremstyle{remark}
\newtheorem{remark}[equation]{Remark}
\newtheorem*{ack}{Acknowledgment}

\numberwithin{equation}{section}
\numberwithin{figure}{section}

\def\norm#1.#2.{\lVert#1\rVert_{#2}}
\def\Norm#1.#2.{\bigl\lVert#1\bigr\rVert_{#2}}
\def\NOrm#1.#2.{\Bigl\lVert#1\Bigr\rVert_{#2}}
\def\NORm#1.#2.{\biggl\lVert#1\biggr\rVert_{#2}}
\def\NORM#1.#2.{\Biggl\lVert#1\Biggr\rVert_{#2}}

\def\ip#1,#2,{\langle #1,#2\rangle}
\def\Ip#1,#2,{\bigl\langle#1,#2\bigr\rangle}
\def\IP#1,#2,{\Bigl\langle#1,#2\Bigr\rangle}

\def\Abs#1{\bigl\lvert#1\bigr\rvert}

\title {Two Weight Inequalities for Discrete Positive Operators}

\author[M.T. Lacey]{Michael T. Lacey}
\address{School of Mathematics \\
Georgia Institute of Technology \\
Atlanta GA 30332 }
\email{lacey@math.gatech.edu}
\thanks{Research supported in part by the NSF, through grant DMS-0456538 and 
DMS 0968499}
\author[E.T. Sawyer]{Eric T. Sawyer}
\address{ Department of Mathematics \& Statistics, McMaster University, 1280
Main Street West, Hamilton, Ontario, Canada L8S 4K1 }
\email{sawyer@mcmaster.ca}
\thanks{Research supported in part by NSERC}
\author[I. Uriarte-Tuero]{Ignacio Uriarte-Tuero}
\address{ Department of Mathematics \\
Michigan State University \\
East Lansing MI 48824}
\email{ignacio@math.msu.edu}
\thanks{Research supported in part by the NSF, through grant DMS-0901524} 
\date{}
\begin{document}

\begin{abstract}  We characterize two weight inequalities for general positive dyadic operators. 
 Let $ \boldsymbol \tau  = \{\tau_Q \;:\; Q\in \mathcal Q\}$ be non-negative constants 
 associated to dyadic cubes, and define a linear operators by 
 \begin{align*}
\operatorname T _{\boldsymbol \tau } f &\coloneqq
\sum _{Q\in \mathcal Q} \tau_Q \cdot \mathbb E _{Q}f \cdot \mathbf 1_{Q} \,. 
\end{align*}
Let $ \sigma , \omega $ be non-negative locally finite weights on $ \mathbb R ^{d}$. We characterize the 
two weight inequalities 
\begin{equation*}
\norm \operatorname T _{\boldsymbol \tau } (f \sigma ). L ^{q} (\omega ). 
\lesssim \norm f. L ^{p} (\sigma ).  \,, \qquad  1<p\le q < \infty \,, 
\end{equation*}
in terms of Sawyer-type testing conditions.  For specific choices of constants $ \tau_Q$, this reduces to the 
two weight fractional integral inequalities of Sawyer \cite{MR930072}. The case of $ p=q=2$, in dimension $ 1$, 
 was characterized by Nazarov-Treil-Volberg \cite{MR2407233}, which result has found several 
interesting applications. 
\end{abstract}

\maketitle

\section{Introduction} 

Our interest is in extensions of the Carleson Embedding Theorem, especially in the discrete setting.  
We recall this well-known Theorem.  Let $ \mathcal Q $ be a choice of dyadic cubes in $ \mathbb R ^{d}$.  
For a cube $ Q$, set 
\begin{equation} \label{e.E}
\mathbb E _{Q}    f \coloneqq   \lvert  Q\rvert ^{-1}  { \int _{Q} f  \; dx }
\end{equation}
Here we are abusing the probabilistic notation for conditional expectation.

\begin{cet}  Let $ \{\tau_Q \;:\; \mathcal Q\}$ be non-negative constants, and let $ 1<p<\infty $.  Define 
\begin{gather}\label{e.Car}  
 \norm \tau_Q . \textup{Car} . \coloneqq 
\sup _{Q\in \mathcal Q} \lvert  Q\rvert  ^{-1} \sum _{\substack{R\in \mathcal Q \\ R\subset Q }}  \tau_R 
  \,, 
\\
C_p \coloneqq \sup _{ \norm f.p.=1} \Bigl[ \sum _{Q\in \mathcal Q} \tau_Q \lvert\mathbb E _Q  f \rvert ^{p} \Bigr] ^{1/p}  \,.
\end{gather}
We have the equivalence $ C_p \simeq \norm \tau_Q . \textup{Car}. ^{1/p} $.  
\end{cet}

We are interested in weighted inequalities, especially two-weight inequalities, and in particular we will give 
discrete extensions of results of Sawyer \cite{MR930072} (also see \cites{MR1437584,MR1175693}) 
and  Nazarov-Treil-Volberg \cite{MR1685781}.  
For the study of such inequalities, 
it is imperative to have \emph{universal} statements, universal in the weight, that can be applied to particular operators.   
By a \emph{weight} we mean a non-negative locally integrable function $ \omega \;:\; \mathbb R ^{d} \to [0,\infty ) $. 
While this is somewhat restrictive, by a limiting procedure, one can pass to more general measures.  For such weights, 
and `nice' sets like cubes $ Q$ we will set 
\begin{equation*}
\omega (Q) \coloneqq \int _{Q} \omega \; dx \,. 
\end{equation*}
  
A first operator that one can construct from a weight is the (dyadic) \emph{maximal function associated to $ w$} 
given by 
\begin{align} \label{e.Mmu}
\operatorname M _{\omega } f (x) & \coloneqq \sup _{Q\in \mathcal Q} \mathbf 1_{Q} (x) \mathbb E ^{\omega }_Q \lvert  f\rvert  \,, 
\\
\mathbb E ^{\omega }_Q f & \coloneqq \omega (Q) ^{-1}  \int _{Q} f \, \omega \; dx \,. 
\end{align}
Here we are extending the definition in \eqref{e.E} to \emph{arbitrary} weights. 
It is a basic fact, proved by exactly the same methods that proves the non-weighted inequality, that we have 

\begin{theorem}\label{t.max} We have the inequalities 
\begin{equation}\label{e.max}
\norm \operatorname M _{\omega } f. L ^{p} (\omega ). \lesssim \norm f. L ^{p} (\omega ).  \,, \qquad 1<p< \infty \,. 
\end{equation}
\end{theorem}

This, by exactly the same proof that proves the Carleson Embedding Theorem, gives us 

\begin{wcet}
 Let $ \{\tau_Q \;:\; \mathcal Q\}$ be non-negative constants,  let $ 1<p<\infty $ and let 
$ w$ be a weight.  Define a weighted version of the Carleson norm by 
\begin{gather}\label{e.wtdCar}
\norm \tau_Q . \textup{Car}, w . \coloneqq 
\sup _{Q\in \mathcal Q} \omega (Q) ^{-1} \sum _{\substack{R\in \mathcal Q \\ R\subset Q }}  \tau_R  \,, 
\\
C_{p,w}  \coloneqq 
\sup _{ \norm f. L ^p (\omega ).=1} \Bigl[ \sum _{Q\in \mathcal Q} \tau_Q \Abs{\mathbb E _Q ^{\omega } f}^{p} \Bigr] ^{1/p} \,. 
\end{gather}
We have the equivalence $ C_{p,w} \simeq \norm \tau_Q . \textup{Car},w. ^{1/p} $.  
\end{wcet}

This is a foundational estimate in the two-weight theory, indeed the only tool needed for the proof of 
the two-weight maximal Theorem of Sawyer \cite{MR676801}. 

We are concerned with the following deep extension, obtained by Nazarov-Treil-Volberg \cite{MR1685781}, of the 
Theorem of Eric Sawyer on two-weight inequalities for fractional integrals \cite{MR930072}.

\begin{sntv}  \label{t.sntv} Let $ \{\tau_Q \;:\; Q\in \mathcal Q\}$ be non-negative constants.  Let $ w, \sigma $ be 
weights.   Define 
\begin{align}\label{e.sntv1}
C_1 ^2 & \coloneqq \sup _{R} \sigma (R)  ^{-1} \int \Bigl[ \sum _{Q\subset R}  \tau_Q \mathbf 1_{Q}  \mathbb E _Q \sigma 
\Bigr] ^2 \omega \,,  
\\ \label{e.sntv2}
C_2 ^2 & \coloneqq \sup _{R} \omega (R)  ^{-1} \int \Bigl[ \sum _{Q\subset R}  \tau_Q \mathbf 1_{Q}  \mathbb E _Q w
\Bigr] ^2 \sigma   
\\
C_3 & \coloneqq 
\sup _{\norm f . L ^{2} (\sigma ).=1} \sup _{\norm g . L ^2 (\omega ).=1}
\sum _{Q\in \mathcal Q} \tau_Q \mathbb E _Q (f \sigma ) \cdot \mathbb E  _Q (g \omega ) \cdot \lvert  Q\rvert 
 \,. 
\end{align}
We have the equivalence $ C_3 \simeq C_1 + C_2 $.  
\end{sntv}

The case of $ \tau_Q = \lvert  Q\rvert ^{\alpha /d}$ for $ 0<\alpha <1$ corresponds to the result of Sawyer.  
Nazarov-Treil-Volberg identified the critical role of this result in two-weight inequalities.   
And it has been subsequently used in the proofs of several results, such as 
\cites{MR2433959,MR2354322,MR2367098,MR1748283,MR1897458} among other papers.  

The Nazarov-Treil-Volberg proof uses the Bellman Function approach.  
Our purpose is to give a new proof of this result, as well as extensions of it.  In particular, our proof 
will work in all dimensions, a result that is new (but expected) in dimensions $ d\ge 2$ and higher.  
We  discuss the general case of $ 1< p\le q < \infty $.  
We also focus on the quantitative versions of these Theorems, as such estimates are important for 
applications.  

\bigskip 

 Let $ \boldsymbol \tau  = \{\tau_Q \;:\; Q\in \mathcal Q\}$ be non-negative constants, and define  linear operators by 
 \begin{align}
\operatorname T _{\boldsymbol \tau } f &\coloneqq \sum _{Q\in \mathcal Q} \tau_Q \cdot \mathbb E _{Q}f \cdot \mathbf
1_{Q} \,,
\\  \label{e.Tin}
\operatorname T _{\boldsymbol \tau, R }  ^{\textup{in}} f &\coloneqq \sum _{\substack{Q\in \mathcal Q\\ Q\subset R}
} \tau_Q \cdot \mathbb E _{Q}f \cdot \mathbf
1_{Q} \,,
\\ \label{e.Tout}
\operatorname T _{\boldsymbol \tau, R }  ^{\textup{out}} f &\coloneqq \sum _{\substack{Q\in \mathcal Q\\ Q\supset R}
} \tau_Q \cdot \mathbb E _{Q}f \cdot \mathbf
1_{Q} 
\end{align}
Here, we are defining the operator $T_\alpha$ and two different `localizations' of $T_\alpha$
corresponding to a cube $R$, one local and the other global.
With these definitions, we have the following equality: 
\begin{equation}\label{e.T=in+out}
 \operatorname T _{\boldsymbol \tau }f (x) = \operatorname T _{\boldsymbol \tau, R }  ^{\textup{in}} f (x) 
 +  \operatorname T _{\boldsymbol \tau, R ^{(1)} }  ^{\textup{out}} f (x') \,, \qquad 
 x\in R\,\  x'\in R ^{(1)} \,. 
\end{equation}
Here and below, we will denote by $ R ^{(1)}$ the `parent' of $ R$:  The minimal dyadic cube that strictly 
contains $ R$.  
Note that the previous Theorem characterizes  the inequality 
\begin{equation}\label{e.T}
\norm \operatorname T _{\boldsymbol \tau} (f \sigma ). L ^{2} (\omega ). \lesssim \norm f. L ^2 (\sigma ).  
\end{equation}

Below, we consider the $ L ^{p} (\sigma ) $ to $ L ^{q} (\omega )$ mapping properties of $ \operatorname T
_{\boldsymbol \tau}$, where $ 1< p\le q < \infty $.  These inequalities are immediately translatable into 
bilinear embedding inequalities.   First, we have the weak-type inequalities.

\begin{theorem}\label{t.dyadicWeakSawyer} Let $\boldsymbol \tau  $ be non-negative constants, and 
$ w, \sigma $ weights. Let $ 1<p\le q < \infty $. 
Define 
\begin{align}\label{e.Weak}
\llbracket \sigma , \omega \rrbracket _{\boldsymbol \tau, p,q}  ^{\textup{Loc}}
&\coloneqq 
\sup _{R\in \mathcal Q}  \omega (R) ^{-1/q'} \norm \operatorname T _{\boldsymbol \tau, R} ^{\textup{in}} (\omega \mathbf 1_{R} ). L ^{p'} (\sigma ). 
\\  \label{e.WEak}
\llbracket \sigma , \omega \rrbracket _{\boldsymbol \tau, p,q}  ^{\textup{Glo}}
&\coloneqq 
\sup _{R\in \mathcal Q}  \omega(R) ^{-1/q'} \norm 
\operatorname T _{\boldsymbol \tau, R} ^{\textup{out}} (\omega \mathbf 1_{R} ). L ^{p'} (\sigma  ). 
\end{align}
We have the equivalence of norms below. 
\begin{align}\label{e.WWeak1}
\norm \operatorname T _{\boldsymbol \tau} (\sigma \cdot ) . L ^{p} (\sigma ) \mapsto L ^{q, \infty } (\omega ). 
&\simeq 
\llbracket \sigma , \omega \rrbracket _{\boldsymbol \tau, p,q}  ^{\textup{Loc}}  \,, \qquad 1<p\le q < \infty 
\\  \label{e.WWeak2}
\norm \operatorname T _{\boldsymbol \tau} (\sigma \cdot ) . L ^{p} (\sigma ) \mapsto L ^{q, \infty } (\omega ). 
&\simeq 
\llbracket \sigma , \omega \rrbracket _{\boldsymbol \tau, p,q}  ^{\textup{Glo}} \,, \qquad 1<p<q<\infty \,. 
\end{align}
Note that the first equivalence holds for  $ p \leq q$, while the second requires a strict inequality.  
\end{theorem}

The `global conditions', in \eqref{e.WWeak2} above and in \eqref{e.strong2} below,
arise from the observations of Gabidzashvili and Kokilashvili \cite{gk}, also see 
\cite{MR1791462}*{Chapter 3} and \cite{MR804116}.  
There is a corresponding, harder,  strong-type characterization.  

\begin{theorem}\label{t.strong}  Under the same assumptions as Theorem~\ref{t.dyadicWeakSawyer} we have 
the equivalences of norms below.  
\begin{align}\label{e.strong1}
\norm \operatorname T _{\boldsymbol \tau} (\sigma \cdot ) . L ^{p} (\sigma ) \mapsto L ^{q } (\omega ). 
&\simeq 
\llbracket \sigma , \omega \rrbracket _{\boldsymbol \tau, p,q}  ^{\textup{Loc}}
+
\llbracket w, \sigma  \rrbracket _{\boldsymbol \tau, q',p'}  ^{\textup{Loc}}
\,, \qquad 1<p\le q < \infty 
\\ \label{e.strong2}
\norm \operatorname T _{\boldsymbol \tau} (\sigma \cdot ) . L ^{p} (\sigma ) \mapsto L ^{q } (\omega ). 
&\simeq 
\llbracket \sigma , \omega \rrbracket _{\boldsymbol \tau, p,q}  ^{\textup{Glo}} 
+
\llbracket w, \sigma  \rrbracket _{\boldsymbol \tau, q',p'.}  ^{\textup{Glo}}
\,, \qquad 1<p<q<\infty \,. 
\end{align}
In particular, the case of \eqref{e.strong1} with $ p=q=2$ is Theorem~\ref{t.sntv}.  
\end{theorem}

We can take $ \sigma  $ and $ w$ to be finite measures and $ f $ a smooth Schwartz function, 
so that there are no convergence issues at any point of the arguments below.  By $ A \lesssim B$ we 
mean $ A< K B$ for an absolute constant $ K$.  By $ A \simeq B$ we mean $ A\lesssim B$ and $ B\lesssim A$.  
We will not try to keep track of constants that depend upon dimension, choices of $ p,\,q$ or  $ \alpha $.

\begin{ack} Two of the  authors completed part of this work while 
participating in a research program at the Centre de Recerca Matem\'atica, 
at the Universitat Aut\`onoma Barcelona,  Spain.  We thank the Centre for their hospitality, and very supportive
environment.
\end{ack}

\section{Proof of the Weak-Type Inequalities} 

Throughout the proofs of both the strong and weak-type results, we will 
suppress the dependence of the operator $ \operatorname T _{\boldsymbol \tau } = \operatorname T$ upon 
$ \boldsymbol \tau = \{\tau_Q\}$.

\subsection{Proof of the necessity of the testing conditions.}
Let  us assume the  weak-type inequality on $ \operatorname T_{} $.  Set 
$
\mathfrak N \coloneqq \norm \operatorname T _{} (\sigma \cdot ) . L ^{p} (\sigma ) \mapsto L ^{q,\infty } (\omega ). < \infty  
$. 
  By duality  for Lorentz spaces, we then have  
\begin{equation*}
\norm \operatorname T_{} (f \cdot \omega ). L ^{p'} (\sigma ).  \le \mathfrak N \norm f. L ^{q',1} (\omega ).  \,.  
\end{equation*}
Apply this inequality to $ f = \mathbf 1_{Q}$ to see that 
\begin{align*}
\norm \operatorname T_{ Q} (\mathbf 1_{Q} \omega ) . L ^{p'} (\sigma ). \le 
\norm \operatorname T_{} (\mathbf 1_{Q} \omega ) . L ^{p'} (\sigma ). 
 \le \mathfrak N \omega (Q) ^{1/q'} \,. 
\end{align*}
Hence $ \llbracket \sigma , \omega \rrbracket _{ p,q}  ^{\textup{Loc}} \le \mathfrak N$.  
 For the global condition, note that 
 \begin{align*}
\omega (Q) ^{-1/q'} \norm \operatorname T _{ Q} ^{\textup{out}}
(\omega \mathbf 1_{Q}) . L ^{p'} (\sigma ).  
& \le \mathfrak N \omega (Q) ^{-1/q'+1/q'} 
= \mathfrak N\,. 
\end{align*}
Hence, $ \llbracket \sigma , \omega \rrbracket _{ p,q}  ^{\textup{Glo}} \le \mathfrak N$.

\subsection{Proof of the weak-type inequality assuming 
\protect{$ \mathfrak L \coloneqq \llbracket \sigma , \omega \rrbracket _{ p,q}  ^{\textup{Loc}} < \infty $.}}  
We consider the proof that the `local testing condition' implies the  weak-type bound for 
$ \operatorname T _{}$.  

Fix $ f\in L ^{p} (\sigma )$,  smooth with  compact support 
and $ \lambda >0$.  We bound the set $ \{ \operatorname T _{}  (f  \sigma )> 2 \lambda \}$. 
Let $ \mathcal Q _{\lambda }$ be the maximal dyadic cubes in $ \{\operatorname T_{}  (f \sigma ) >
\lambda  \}$ which also intersect the set $ \{\operatorname T_{} (f \sigma )> 2 \lambda \}$. 

Let $  Q ^{(1)}$ denote the parent of a dyadic cube. For  fixed $ Q_0 \in \mathcal Q _{\lambda }$, 
we must have that $Q _0 ^{(1)} $ contains a point $ z$ with $ \operatorname T_{} (f \sigma )
(z)< \lambda $. It follows that 
\begin{align*}
\lambda & >  \operatorname T_{} (f \sigma )
(z) 
\ge \operatorname T _{Q_0^{(1)}} ^{\textup{out}} (f \sigma ) \,.  
\end{align*}
From this, we must have 
\begin{align} \label{e.FI-weak1}
\lambda \le \operatorname T _{Q_0} ^{ \textup{in}} (f \sigma ) (x)\,, 
\qquad x \in Q_0 \cap  \{\operatorname T_{} (f \sigma ) (x)> 2 \lambda \} \,. 
\end{align}
This represents an important localization of the operation $ \operatorname T_{} (f \sigma )$.

Note that we can estimate 
\begin{align}
M & \coloneqq   \sum _{Q\in \mathcal Q _{\lambda }} 
\Bigl[ \frac 1 {\omega (Q)} \int _{Q} \operatorname T_{ Q} ^\textup{in} (f \sigma )\omega  \; dx \Bigr] ^{q} 
\omega (Q) 
\\
& \lesssim 
\sum _{Q\in \mathcal Q _{\lambda }} 
\Bigl[  \int _{Q} f \sigma  \operatorname T_{ Q} ^\textup{in} (\mathbf 1_{Q} \omega ) \; dx \Bigr] ^{q} 
\omega (Q) ^{1-q}
\\
& \lesssim  
\mathfrak L ^{q}
\sum _{Q\in \mathcal Q _{\lambda }}  \Bigl[ \int _{Q} \lvert  f\rvert ^{p} \sigma   \Bigr] ^{q/p} 
\omega (Q) ^{q/q'+1-q}  
\\
& \lesssim 
\mathfrak L ^{q}
\Bigl[\sum _{Q\in \mathcal Q _{\lambda }}   \int _{Q} \lvert  f\rvert ^{p} \sigma   \Bigr] ^{q/p} 
& (p\le q)
\\  \label{e.M<}
& \lesssim 
\mathfrak L ^{q}
\norm f. L ^{p} (\sigma ). ^{q} \,. 
\end{align}
Note that we have used duality to move the (self-dual) operator $ \operatorname T _{\alpha } ^{\textup{in}}$ over to 
the simpler function.

To complete the proof, we will split $ \mathcal Q_\lambda $ into subcollections $ \mathcal E$ and $ \mathcal F$, 
where $ \mathcal E$ consists of those cubes which are `empty' of the set $ \{\operatorname T_{} (f \sigma )
> 2 \lambda \}$, precisely for $  \eta = 2 ^{-q-1}$ 
\begin{equation*}
\mathcal E \coloneqq 
\bigl\{ Q\in \mathcal Q _{\lambda } \;:\; 
\omega ( Q \cap \{\operatorname T_{} (f \sigma )
> 2 \lambda \})< \eta \omega (Q) 
\bigr\}\,, 
\end{equation*}
and $ \mathcal F= \mathcal Q _{\lambda }- \mathcal E $. 
And to conclude the proof, we can estimate, using \eqref{e.M<}, 
\begin{align*}
(2\lambda) ^{q} \omega (\operatorname T_{} (f \sigma ) > 2 \lambda ) 
& 
\le  \eta (2 \lambda ) ^{q} \sum _{Q\in \mathcal E} \omega (Q) 
+ \eta ^{-q} M 
\\
& \le 
\eta 2 ^{q}  \lambda ^{q} \omega (\operatorname T_{} (f \sigma ) >  \lambda ) 
+  C  \eta ^{-q}
\mathfrak L ^{q}
\norm f. L ^{p} (\sigma ). ^{q} \,. 
\end{align*}
Take $ \lambda $ so that the left-hand side of this inequality  
is close to maximal. (The supremum is a finite number by assumption.) 
By choice of $ \eta $, this proves the estimate.

\subsection{Proof of  the weak-type inequality assuming 
\protect{$\mathfrak G \coloneqq \llbracket \sigma , \omega \rrbracket _{ p,q}  ^{\textup{Glo}}< \infty $.}
}
We show that the   `global testing condition' implies the weak-type inequality for the  
fractional integral operator, when $ p<q$.
This proof will depend upon a (clever) comparison to a maximal function.    
We proceed with the initial steps of the previous proof, up until \eqref{e.FI-weak1}.

We rewrite the  sum in \eqref{e.FI-weak1} in a way that permits our application of the 
`global' testing condition. 
Inductively define $ Q_k$ containing $ x$ as follows. The cube $ Q_0$  and $ x$ are  as \eqref{e.FI-weak1} above, 
and given $ Q_k\subset Q_0$, take $ Q_{k+1}$ to be the maximal dyadic cube containing $ x$ that satisfies 
$ \omega (Q _{k+1})\le \tfrac 12 \omega (Q_k)$.  Then, we have, continuing from \eqref{e.FI-weak1}, 
\begin{align}\label{e.FI-weak2}
\lambda & \le  \sum _{k=0} ^{\infty } 
\sum _{\substack{Q \;:\; x\in Q\\ Q _{k+1}\subsetneqq  Q\subset Q_k }}  \tau_Q \lvert  Q\rvert ^{ -1} 
\int _{Q} f \; \sigma dy 
\\ & 
\le \sum _{k=0} ^{\infty }  \int _{Q_0}  \Biggl\{ \sum _{\substack{Q \;:\; x\in Q\\ Q _{k+1}\subsetneqq  Q\subset Q_k }}   
\tau_Q \lvert  Q\rvert ^{ -1}  \mathbf 1_{Q}\Biggr\} f \; \sigma dy
\\
& \lesssim 
\sum _{k=0} ^{\infty }  \int _{Q_0}   \omega (Q^{(1)} _{k+1}) ^{-1} \operatorname T _{ Q^{(1)} _{k+1}} ^{\textup{out}}
(\omega \mathbf 1_{Q^{(1)}_{k+1}})\cdot 
(f \mathbf 1_{Q _{k}}) \; \sigma dy    
\\
& \lesssim 
\sum _{k=0} ^{\infty }  \omega (Q^{(1)} _{k+1}) ^{-1}  \norm  \operatorname T _{Q^{(1)} _{k+1}} ^{\textup{out}}
(\omega \mathbf 1_{ Q^{(1)}_{k+1}}) . L ^{p'} (\sigma ).  
\Bigl(\int _{ Q _k} f ^{p} \sigma  \Bigr) ^{1/p} 
\\  
& \lesssim 
\mathfrak G
\sum _{k=0} ^{\infty }  \omega ( Q^{(1)} _{k+1} ) ^{-1/q}  
\Biggl(   \int _{ Q_k} f ^{p} \; \sigma \Biggr) ^{1/p} 
\\  
& \lesssim 
\mathfrak G
\sum _{k=0} ^{\infty }    \omega (Q_k) ^{1/p}\omega ( Q^{(1)} _{k+1} ) ^{-1/q}  
\Biggl(  \omega (Q_k) ^{-1} \int _{ Q_k} f ^{p} \; \sigma \Biggr) ^{1/p} 
\\  \label{e.pass}
& \lesssim 
\mathfrak G
\omega (Q_0) ^{1/p-1/q} \overline{\operatorname M} f (x) \,. 
\end{align}
In the last inequality,  we define the maximal function $ \overline{\operatorname M}$ as follows. 
\begin{equation*}
\overline{\operatorname M} f (x) \coloneqq  
\sup _{Q \;:\; Q \subset Q_0} \mathbf 1_{Q} (x) \Bigl[ \omega (Q) ^{-1} \int _{Q} f ^{p} \; \sigma  \Bigr] ^{1/p}\,. 
\end{equation*}
This is a localized maximal function, with both weights involved in the definition.    
In passing to \eqref{e.pass}, we should note that we are certainly using the strict inequality $ p<q$: 
By construction, $ \omega (Q^{(1)} _{k+1})\ge \tfrac 12 \omega (Q_k)$, so that 
\begin{align*}
\sum _{k=0} ^{\infty }    \omega (Q_k) ^{1/p}\omega ( Q^{(1)} _{k+1} ) ^{-1/q}  
& \lesssim 
\sum _{k=0} ^{\infty }    \omega (Q_k) ^{1/p-1/q} \lesssim  \omega (Q_0) ^{1/p-1/q} \,. 
\end{align*}

\medskip 

The conclusion of these calculations is that for maximal dyadic $ Q_0\subset \{\operatorname T_{} (f \sigma )>
\lambda \}$,  and $ x\in Q_0\cap \{\operatorname T_{} (f \sigma )> 2 \lambda \}$, we have 
\begin{equation*}
\lambda \le c 
\mathfrak G 
\omega (Q_0) ^{1/p-1/q} \overline{\operatorname M} f (x) \,.  
\end{equation*}
We proceed with an estimate for $ \omega ( Q_0 \cap \{\operatorname T_{} (f \sigma )> 2 \lambda \} )$.  

Take $ \mathcal P_0$ to be the maximal dyadic cubes $ Q\subset Q_0$ so that 
\begin{align*}
\lambda  & \le c 
\mathfrak G 
\omega (Q_0) ^{1/p-1/q} \Biggl[ \omega (Q) ^{-1} \int _{Q} f ^{p} \; \sigma  \Biggr] ^{1/p} \,, 	  
\\
\noalign{\noindent {or,  what is the same}}
\omega (Q) & \le c 
\mathfrak G^p 
\lambda ^{-p} \omega (Q_0)  ^{1-1/q} \int _{Q} f ^{p} \; \sigma  \,. 
\end{align*}
And this permits us to estimate 
\begin{align*}
\lambda ^{q} \omega \bigl(  Q_0 \cap \{\operatorname T_{} (f \sigma )> 2 \lambda \} \bigr) 
& \le \lambda ^{q} \sum _{Q\in \mathcal P_0} \omega (Q) 
\\
& \lesssim  
\mathfrak G^p 
\lambda ^{q-p} \omega (Q_0) ^{1-p/q} \sum _{Q\in \mathcal Q_0} \int _{Q} f ^{p}  \; \sigma 
\\
& \lesssim 
\mathfrak G^p 
\lambda ^{q-p} \omega(Q_0) ^{1-p/q}  \int _{Q_0} f ^{p}  \; \sigma  \,. 
\end{align*}

\medskip

We have to this moment been working with a single maximal $ Q_0\subset \{\operatorname T_{} (f \sigma )>\lambda
\}$ which also meets the set 
$ \{\operatorname T_{} (f \sigma )>
2\lambda \}$.
Let $ \mathcal Q_0$ be the collection of all such $ Q_0$.  We can estimate 
\begin{align}
(2 \lambda) ^{q} \omega (\operatorname T _{\alpha } (f \sigma )> 2 \lambda ) 
&\lesssim  
\mathfrak G^p 
\lambda ^{q-p} 
\sum _{Q_0\in \mathcal Q_0} \omega (Q_0) ^{1-p/q} \int _{Q_0} f ^{p} \; \sigma  
\\
& \lesssim \mathfrak G^p  
\lambda ^{q-p}  
\Biggl[\sum _{Q_0\in \mathcal Q_0} \omega (Q_0) \Biggr] ^{1-p/q} 
\cdot 
\Biggl[ \sum _{Q_0\in \mathcal Q_0} \Bigl(\int _{Q_0} f ^{p} \sigma  \Bigr) ^{q/p} \Biggr] ^{p/q} 
\\  \label{e.WFI-i}
& \lesssim  \mathfrak G^p 
\bigl[ \lambda ^{q} \omega (\operatorname T_{} (f \sigma )> \lambda ) \bigr] ^{1- p/q} 
\int  f ^{p} \; \sigma  \,. 
\end{align}

  Apply \eqref{e.WFI-i} with a choice of $ \lambda  $ so that the 
left-hand side is close to maximal.  It follows that we have 
\begin{equation*}
[\lambda ^{q} \omega (\operatorname T _{\alpha } (f \sigma )> 2 \lambda) ] ^{p/q} 
\lesssim \mathfrak G^p 
\int  f ^{p} \; \sigma  \,. 
\end{equation*}
And this completes the proof.

\section{Proof of Sawyer's Two Weight Norm Result} 

\subsection{Linearizations of Maximal Functions}

The maximal theorem Theorem~\ref{t.max}, giving universal bounds on the maximal function $ \operatorname M _{\omega }$,
will be an essential tool, arising in proof 
of the sufficiency of the testing conditions below.  It will arise in a `linearized' form. 
By this we mean the usual way to pass from a sub-linear maximal operator to a linear one, 
which for $ \operatorname M _{\omega }$ means the following.  

Let $ \{E (Q) \;:\; Q\in \mathcal Q\}$ be any selection of measurable disjoint sets $ E (Q)\subset Q$ indexed 
by the dyadic cubes. Define a 
corresponding linear operator $ \operatorname L$ by 
\begin{equation}\label{e.Lop}
\operatorname L f (x) \coloneqq \sum _{Q\in \mathcal Q} \mathbf 1_{E (Q)} (x) \mathbb E ^{\omega } _{Q} f \,. 
\end{equation}
Then, \eqref{e.max} is equivalent to the bound $ \norm \operatorname L f . L ^{p} (\omega ). \lesssim \norm f. L ^{p} (\omega ).$ 
with implied constant independent of $ w$ and the sets $ \{E (Q) \;:\; Q\in \mathcal Q\}$.

\subsection{Initial Considerations. Whitney Decomposition.}
In this proof we will only explicitly use the  `local' testing conditions, which is sufficient to 
deduce the Theorem as the previous arguments 
show that the `local' and `global' conditions are equivalent, in the case of $ 1<p<q<\infty $.  
Let us set 
\begin{align} \label{e.L}
\mathfrak L           \coloneqq \llbracket \sigma , \omega \rrbracket _{ p,q}  ^{\textup{Loc}} \,, 
\qquad &
\mathfrak L _{\ast } \coloneqq \llbracket  w, \sigma  \rrbracket _{ q',p'}  ^{\textup{Loc}} \,. 
\end{align}
There is a very useful strengthening of the assumption that we can exploit, due to the fact that we have 
already proved the weak-type results, namely Theorem~\ref{t.dyadicWeakSawyer}.  Due to 
\eqref{e.WWeak1}, we have 
\begin{equation}\label{e.L'}
\sup _{Q\in \mathcal Q}  \omega (Q) ^{-1/q'} \norm \operatorname T _{  } (\mathbf 1_{Q} \omega ). L ^{p'} (\sigma ). 
\lesssim \mathfrak L \,. 
\end{equation}

We take $ f$ to be a finite combination of indicators of dyadic cubes. 
We work with the sets 
$
\Omega _k = \{ \operatorname T_{}  (f \sigma )> 2 ^{k}\} 
$, 
which are open, and begin by making a Whitney-style decomposition of all of these sets. 

Let $ Q ^{(1)}$ denote the parent of $ Q$, and inductively define $ Q ^{(j+1)}= (Q ^{(j)})^{(1)}$.  
For an integer $ \rho \ge 2$, we should choose collections $ \mathcal Q _k$ of disjoint dyadic cubes so that these 
several conditions are met. 
\begin{align}\label{e.dc}
\Omega _k = \bigcup _{Q\in \mathcal Q_k} Q  & & \textup{(disjoint cover)}
\\ \label{e.Whit}
Q^{(\rho)} \subset \Omega _k\,, \ Q^{(\rho+1)}\cap \Omega _k ^{c} \neq \emptyset & & \textup{(Whitney condition)}
\\ \label{e.fo}
\sum _{Q\in \mathcal Q_k} \mathbf 1_{Q^{(\rho)}} \lesssim \mathbf 1_{\Omega _k} 
& &\textup{(finite overlap)}
\\ \label{e.crowd}
\sup _{Q\in \mathcal Q_k} {} ^{\sharp} \{Q'\in \mathcal Q_k \;:\; Q'\cap Q ^{ (\rho)}\neq \emptyset \} \lesssim 1 \,, 
& &  \textup{(crowd control)}
\\ \label{e.nested}
Q\in \mathcal Q_k\,,\ Q'\in \mathcal Q_l \,,\ Q\subsetneqq Q'\quad \textup{implies } \quad 
k> l \,. & &\textup{(nested property)}
\end{align}
We will prove this for arbitrary $ \rho $, but take $ \rho =1$ in the proof. 

\begin{proof}
Take $ \mathcal Q_k$ to be the maximal dyadic cubes $ Q\subset \Omega _k$ which satisfy 
\eqref{e.Whit}. Such cubes are disjoint and \eqref{e.dc} holds.   As the sets $ \Omega _k$ are themselves nested, 
\eqref{e.nested} holds.   

\smallskip 
Let us  show that \eqref{e.fo} holds.
Note that holding the volume of the cubes constant we have 
\begin{equation*}
\sum _{\lvert  Q\rvert= 1} \mathbf 1_{Q^{(\rho)}} \le 2 ^{\rho d}
\end{equation*}
where $ d$ is the dimension.  So if we take an integer $ \rho $, and assume that for some $ k$ and $ x\in \mathbb R ^{d}$
\begin{equation*}
\sum _{Q\in \mathcal Q_k} \mathbf 1_{Q^{(\rho)}} (x) \ge 8 \cdot 2 ^{(\rho+1) d}\,, 
\end{equation*}
then we can choose $ Q, R\in \mathcal Q_k$ with $ x\in Q^{(\rho)}\cap R^{(\rho)}$ and the side-length of $ R$ satisfies 
$ \lvert  R\rvert ^{1/d} \le 2 ^{-3} \lvert  Q\rvert ^{1/d}  $.   But then it will follow that 
$ R^{(\rho+1)}\subset Q^{(\rho)} $.  We thus see that $ R^{(\rho+1)}$ does not meet 
$ \Omega _k ^{c}$, which is a contradiction. 

\smallskip 
Let us see that \eqref{e.crowd} holds.  Fix $ Q\in \mathcal Q_k$.  If we had $ Q'\supsetneqq Q ^{(\rho)}$ 
for any $ Q'\in \mathcal Q_k$, we would violate \eqref{e.Whit}.  Thus, we must have $ Q'\subset Q ^{(\rho)}$, 
and these cubes Q' are disjoint. Suppose that there were more than $ 2 ^{\rho +2}$ in number.  Then, there would have 
to be a $ Q' \subset Q ^{(\rho )}$ with $ \lvert  Q'\rvert\le 2 ^{-\rho -1} \lvert  Q ^{(\rho )}\rvert  $.  
That is, $ (Q') ^{(\rho +1)}\subset Q ^{(\rho )}$, violating the Whitney condition \eqref{e.Whit}.

\end{proof}

Let us comment on a subtle point that enters in a decisive way at the end of the proof, see Proposition~\ref{p.bounded}.  
A given cube $ Q$ can be a member of an unbounded number of $ \mathcal Q_k$. Namely, there 
are integers $ K _{-} (Q)\le K _{+} (Q)$ so that 
\begin{equation}\label{e.many}
Q\in \mathcal Q_k\,, \qquad K _{-} (Q)\le k \le  K _{+} (Q)\,, 
\end{equation}
and there is no \emph{a priori} upper bound on $ K _{+} (Q)- K _{-} (Q)$.

\subsection{Maximum Principle. Decomposition of \protect{$ \norm \operatorname T   f. L ^{p} (\omega ). ^{p}$.}}
There is an important maximum principle  which will serve to further localize the operation $ \operatorname T  $.  
For all $ k$ and $ Q\in \mathcal Q_k$ we have 
\begin{equation}\label{e.maxPrin}
 \max\bigl\{ 
 \operatorname T_{Q^{(1)}} ^{\textup{out}} (f \mathbf 1_{Q ^{(2)}} \sigma ) (x)
 \,,\, \operatorname T_{}  (\mathbf 1_{(Q^{(2)}) ^{c}} f \sigma ) (x) 
\bigr\} \le 2 ^{k} 
\,, \qquad x\in Q\,.
\end{equation}
\begin{proof}
We can choose 
 $ z\in Q^{(2)} \cap \Omega _{k} ^{c}$, which exists by \eqref{e.Whit}. Then, for $ x\in Q$
\begin{equation*}
\operatorname T_{}  (\mathbf 1_{(Q^{(2)}) ^{c}} f \sigma ) (x) 
= \operatorname T_{Q^{(1)}} ^{\textup{out}} (\mathbf 1_{(Q^{(2)}) ^{c}} f \sigma ) (x)  
\le \operatorname T_{} (f \sigma ) (z) \le 
2 ^{k}\,. 
\end{equation*}
Also, it is clear that 
$ \operatorname T_{Q^{(1)}} ^{\textup{out}} (f \mathbf 1_{Q ^{(2)}} \sigma ) (x) 
\le \operatorname T_{} (f \sigma ) (z) \le 2 ^{k}$.  
\end{proof}

Let us set  $ m=5$. We will use this integer throughout the remainder of the proof. 
Define the  sets 
\begin{equation}\label{e.EQ}
E_k(Q) \coloneqq Q \cap (\Omega _{k+m-1} - \Omega _{k+m})\,, \qquad Q\in \mathcal Q_k \,. 
\end{equation}
It is required to include the subscript $ k$ here, and in other places below, due to \eqref{e.many}. 
See Figure~\ref{f.EkQ} for an illustration of this set.

\begin{figure}
\begin{center}
\includegraphics[scale=.8]{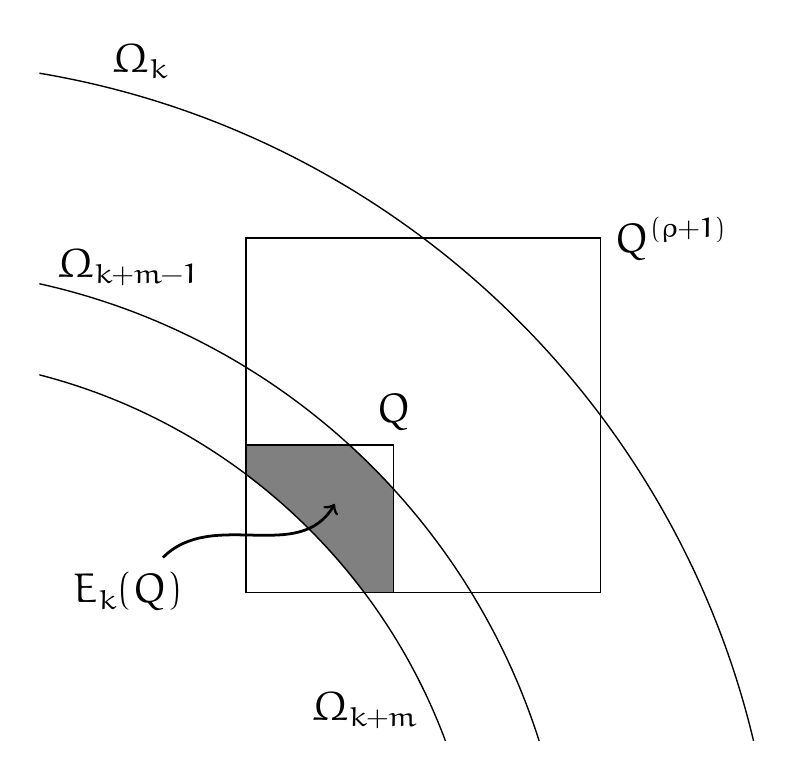}
\end{center}  
\caption{The set $ E_k (Q)$.}   
\label{f.EkQ}
\end{figure}

Now, the Maximum Principle, the equality \eqref{e.T=in+out}, and choice of $ m$ gives us for $ x\in E_k(Q)$
\begin{align} \label{e.MPfig}
 \operatorname T_{Q^{(1)}} ^{\textup{in}}  (\mathbf 1_{Q^{(2)}} f \sigma ) (x) 
& = \operatorname T_{}  (f \sigma ) (x) - 
 \operatorname T_{Q^{(1)}} ^{\textup{out}} (f \mathbf 1_{Q ^{(2)}} \sigma ) (x)
 - \operatorname T_{}  (\mathbf 1_{(Q^{(2)}) ^{c}} f \sigma ) (x) 
\\
& \ge 2 ^{k+m-1} - 2 ^{k+1}\ge 2 ^{k} \,. 
\end{align}
We should make one more observation.  By the definition of $ \operatorname T ^{\textup{in}}$,
we have 
\begin{equation*}
\operatorname T_{Q^{(1)}} ^{\textup{in}}  (\mathbf 1_{Q^{(2)}} f \sigma ) (x) 
= \operatorname T_{Q^{(1)}} ^{\textup{in}}  (\mathbf 1_{Q^{(1)}} f \sigma ) (x) \,, 
\qquad x\in Q \,. 
\end{equation*}
On the right, we replace the cube $ Q ^{(2)}$ inside $ \operatorname T $ with $ Q ^{(1 )}$. 
This will be useful for us as it will, at a moment below, place the crowd control principle \eqref{e.crowd}
at our disposal.


This permits us the following calculation which is basic to the organization of the proof.  
\begin{align}
2 ^{k}\omega (E_k(Q)) & \le \int _{E_k(Q)}
\operatorname T_{Q^{(1)}} ^{\textup{in}} 
(\mathbf 1_{Q^{(1)} } f \sigma )\;  w 
\\
& = \int _ {Q^{(1)}} f \cdot\operatorname T_{Q^{(1)}} ^{\textup{in}} 
(\mathbf 1_{E_k(Q)} \omega ) \; \sigma  
\\
& =  \alpha_k(Q) + \beta_k(Q) \,, 
\\ \label{e.sigma}
\alpha_k(Q) &\coloneqq 
\int _{Q^{(1)}\backslash \Omega _{k+m}} f \cdot 
\operatorname T_{Q^{(1)}} ^{\textup{in}}  (\mathbf 1_{E_k(Q)} \omega ) \; \sigma \,, 
\\ \label{e.tau}
\beta_k(Q) & \coloneqq 
\int _{Q^{(1)} \cap \Omega _{k+m}} f \cdot 
\operatorname T_{Q^{(1)}} ^{\textup{in}}  (\mathbf 1_{E_k(Q)} \omega ) \; \sigma \,.
\end{align}
It is the term $ \beta _k (Q)$ that leads to the (much) harder term.

And then, we can estimate 
\begin{align}
\int \lvert  \operatorname T_{}  (f \sigma )\rvert ^{q} \;\omega
& \le 2 ^{mq} \sum _{k=-\infty } ^{\infty }   2 ^{kq}\omega (\Omega _{k+m-1}- \Omega _{k+m}) 
\\
& = 2 ^{mq} \sum _{k=-\infty } ^{\infty } \sum _{Q\in \mathcal Q _k} 2 ^{kq}\omega ( E_k(Q))  
\\ \label{e.SSS}
& = 2 ^{mq} \sum _{j=1} ^{3} S_j \,. 
\end{align}
The last three sums are defined by a choice of $ 0< \eta <1$ and 
\begin{align}\label{e.S}
S_j & \coloneqq \sum _{k=-\infty } ^{\infty } \sum _{Q\in \mathcal Q_k ^{j}} 2 ^{kq}\omega ( E_k(Q))  \,, 
\qquad j=1,2,3\,,
\\   \label{e.Q1}
\mathcal Q _k ^{1} & \coloneqq  
\{ Q\in \mathcal Q_k \;:\;  \omega (E_k(Q))\le \eta \omega (Q) \}\,, 
\\ \label{e.Q2}
\mathcal Q _k ^{2} & \coloneqq  
\{ Q\in \mathcal Q_k \;:\;  \omega (E_k(Q))> \eta \omega (Q)\,,\ 
\alpha_k(Q)> \beta_k(Q) 
\}\,, 
\\ \label{e.Q3}  
\mathcal Q_k ^{3} & \coloneqq \mathcal Q_k-\mathcal Q_k ^{1} - \mathcal Q_k ^{2} \,. 
\end{align}
Here, let us note that $ \mathcal Q _k ^{1}$ consists of those $ Q \in \mathcal Q_k$ such that 
$ E _{k} (Q)$ is `empty,' and these terms will be handled much as they were in the weak-type argument. 
Using the notation of \eqref{e.many},  observe that 
\begin{equation}\label{e.many<}
^{\sharp} \{ K _{-} (Q)\le k\le K _{+} (Q) \;:\;  Q \in \mathcal Q_k \backslash \mathcal Q _{k} ^{1}\} \le  \eta ^{-1}   \,.  
\end{equation}
This follows from the definition of $ \mathcal Q ^{1}_k$, and that the sets $ E _{k } (Q)$
are pairwise disjoint in $ k$. 
This point enters in Proposition~\ref{p.bounded} below.

We will bound each of the $ S_j$ in turn.  In fact, recalling \eqref{e.L}, we show that 
\begin{align}\label{e.Q1<}
S_1 & \lesssim \eta \norm   \operatorname T_{}  (f \sigma ). L ^{q} (\omega ). ^{q}
\\ \label{e.Q2<} 
S_2  &\lesssim \eta ^{-q}  \mathfrak L  ^{q} \norm  f . L ^{p}(\sigma). ^{q}  
\\ \label{e.Q3<}
S_3 & \lesssim \eta ^{-q}  \bigl[ \mathfrak L ^{q} +\mathfrak L  _{\ast } ^{q}\bigr] 
\norm  f\rvert . L ^{p}(\sigma). ^{q} \,.
\end{align}
Thus, the term $ S_2$ requires the weak-type testing condition, while $ S_3$ requires both testing conditions.

\begin{figure}
\begin{center}
\includegraphics[scale=1]{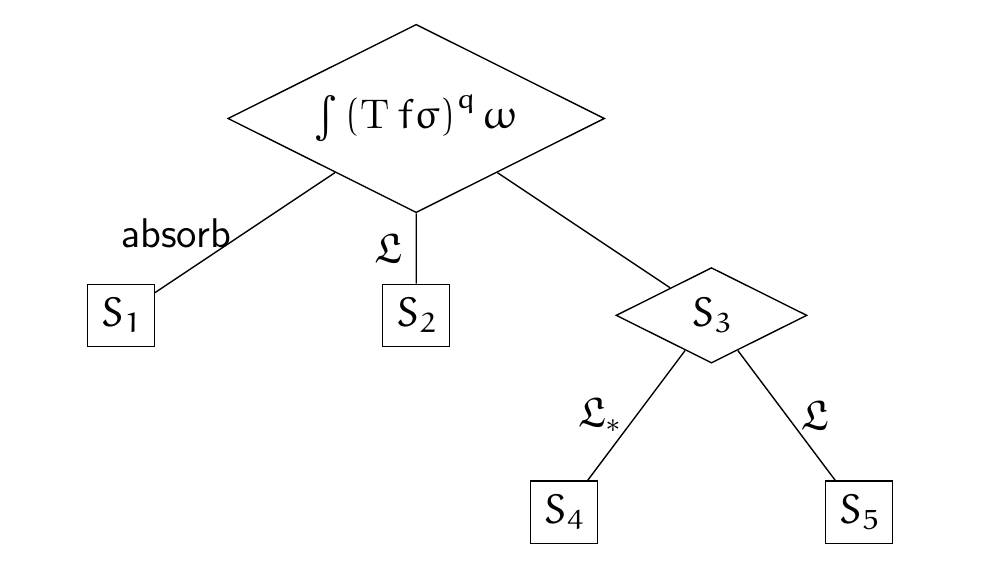}
\end{center}
\caption{Schematic Tree for the proof of the strong type inequality.}  
\label{f.schematic}
\end{figure}

This permits us to estimate 
\begin{equation*}
 \int \lvert  \operatorname T_{}  (f \sigma )\rvert ^{q} \;\omega  
 \lesssim \eta  \cdot 
\norm   \operatorname T_{}  (f \sigma ). L ^{q} (\omega ). ^{q}+ 
  \eta ^{-q} \bigl[ \mathfrak L ^{q} + \mathfrak L _{\ast } ^{q}\bigr] \cdot 
 \norm  f\rvert . L ^{p}(\sigma). ^{q} \,. 
\end{equation*}
The selection of $ \eta $ is independent of the selection of $ m$ (which is after all specified). 
So for small $ 0<\eta <1 $, we can absorb the first term on the right into the left-hand side, 
proving our Theorem.

We include a  schematic tree of the 
proof in Figure~\ref{f.schematic}.   Concerning this figure we make these comments. 
\begin{itemize}
\item  Terms in diamonds are further decomposed, 
while 
those in boxes are final estimates. 
\item 
The testing conditions used to control each final estimate are indicated on the 
edges. The label `absorb' on $ S_1$ indicates that this term is absorbed into the main term.
\item $ S_3$ is the hard term, requiring further decomposition, guided by 
the  introduction of the `principal cubes,' see Remark~\ref{r.pc}, 
and some delicate combinatorics, see Proposition~\ref{p.bounded}.
\end{itemize}

\subsection{Two Easy Estimates.}

Now, the estimates \eqref{e.Q1<} for $ S_1$ and \eqref{e.Q2<} for $ S_2$ are reasonably straight forward, 
but more involved for $ S_3$.  Let us bound $ S_1$.  By the definition in 
\eqref{e.Q1},   the sets $ E_k(Q)$ are nearly empty.  
\begin{align}
S_1 & =  \sum _{k=-\infty } ^{\infty } \sum _{Q\in \mathcal Q_k ^{1} } 2 ^{kq} \omega(E_k(Q)) 
\\
& \le \eta 
\sum _{k=-\infty } ^{\infty } \sum _{Q\in \mathcal Q_k ^{1} } 2 ^{kq} \omega(Q)  
\\ \label{e.S1<}
& \le \eta 
\sum _{k=-\infty } ^{\infty } 2 ^{kq} \omega ( \{\operatorname T_{}  (f \sigma )> 2 ^{k}\}) 
\\
& \lesssim \eta \cdot \norm   \operatorname T_{}  (f \sigma ). L ^{q} (\omega ). ^{q}
\end{align}
Here, we have used the condition \eqref{e.dc}. 
\smallskip 

Let us turn to $ S_2$. 
The defining condition in \eqref{e.Q2} is that 
\begin{align*}
\eta  2 ^{k}\omega (Q) & \le 2 ^{k} \omega (E_k(Q)) 
\\
& \lesssim  \alpha_k(Q)
\\
& = \int _{Q^{(1)}\backslash \Omega _{k+m} } f \cdot  \operatorname T_{Q^{(1)}} ^{\textup{in}}  (\mathbf 1_{E_k(Q)} \omega ) \; \sigma 
\\
& \le 
\Bigl[\int _{Q^{(1)}\backslash \Omega _{k+m} } f ^{p} \; \sigma   \Bigr] ^{1/p} 
\cdot 
\Bigl[ \int _{Q^{(1)}} \bigl( \operatorname T_{Q^{(1)}} ^{\textup{in}}  (\mathbf 1_{E_k(Q)} \omega ) \bigr) ^{p' }\; \sigma \Bigr] ^{1/p'} 
\\
& \le \mathfrak L 
\Bigl[\int _{Q^{(1)}\backslash \Omega _{k+m} } f ^{p} \; \sigma   \Bigr] ^{1/p} 
\cdot  \omega (Q) ^{1/q'} \,. 
\end{align*}
We have used the weak-type testing condition, and in particular \eqref{e.L'}. The estimate we use from this is 
\begin{equation*}
2 ^{k} \lesssim \mathfrak L  
\eta ^{-1} \omega (Q) ^{-1/q}  \Bigl[\int _{Q^{(1)}\backslash \Omega _{k+m} } f ^{p} \; \sigma   \Bigr]
^{1/p} \,. 
\end{equation*}

Using this estimate, we can finish the estimate for $ S_2$.  
\begin{align}
S_2 & = 
\sum _{k=-\infty } ^{\infty } \sum _{Q\in \mathcal Q_k ^{2}}   2 ^{kq} \omega(E_k(Q))
\\
& \lesssim  \eta ^{-q} \mathfrak L  ^{q}
\sum _{k=-\infty } ^{\infty } \sum _{Q\in \mathcal Q_k ^{2}}  
\frac {\omega (E_k(Q))} {\omega (Q)}
\Bigl[\int _{Q^{(2)}\backslash \Omega _{k+m} } f ^{p} \; \sigma   \Bigr] ^{q/p} 
\\ \label{e.S2<}
& \lesssim \eta ^{-q} 
 \mathfrak L ^{q}
\Bigl[
\int f ^{p} 
\sum _{k=-\infty } ^{\infty } \sum _{Q\in \mathcal Q_k ^{2}}  \mathbf 1_{Q^{(2)}\backslash \Omega _{k+m} } 
\;  \sigma \Bigr] ^{q/p} \qquad \qquad  (q/p\ge 1)
\\
& \lesssim \eta ^{-q}
 \mathfrak L  ^{q} 
 \Bigl[
\int f ^{p} 
\;  \sigma \Bigr] ^{q/p}\,. 
\end{align}
Here, the $ \Omega _{k}$ are decreasing sets, so the sum over $ k$ above is bounded by $ m=5$.  
This completes the proof of \eqref{e.Q2<} for $ S_2$. 

\subsection{The Difficult Case, Part 1.}
We turn to the last and most difficult case, namely the estimate for \eqref{e.Q3<}.  
This subsection will introduce the essential tools for the analysis of this term, namely the 
collections $ \mathcal R_k (Q)$ in \eqref{e.RQ} and the `principal cubes' construction, see the 
paragraph around \eqref{e.PC1}.

For integers $ 0\le M<m$ we will show that 
\begin{equation}\label{e.3MN}
S _{3,M} \coloneqq 
\sum _{\substack{k \equiv M \; \textup{mod} m}} 
\sum _{Q\in \mathcal Q_k ^{3}} 2 ^{kq} \omega(E_k(Q)) \lesssim 
\bigl\{\mathfrak L + \mathfrak L _{\ast }  \bigr\}
^{q} 
\eta ^{-q}\norm f. L ^{p} (\sigma ) . ^{q} \,, 
\end{equation}
where the implied constant is independent of $  M$ and $ N$.  Summing over 
$ M$  will prove \eqref{e.Q2<} for $ S_3$. 
It is the standing assumption for the remainder of the proof of \eqref{e.3MN} that $ k\equiv M \mod m$.

This collection of cubes is important for us. 
\begin{equation} \label{e.RQ}
\mathcal R_k (Q) \coloneqq 
\{R\in \mathcal Q _{k+m} \;:\; Q^{(1)} \cap R\neq \emptyset \}\,, \qquad Q\in \mathcal Q_k ^{3}\,. 
\end{equation}
Recall that the  set we are integrating over in $ \beta _{k} (Q)$
 is $ Q ^{ (1 )}\cap \Omega _{k+m}$, \eqref{e.tau}.  Now, 
 for $ R\in \mathcal R _{k} (Q)$, we have 
$ R\subset Q ^{(1 )}$.  Indeed, if this is not the case, we have 
$ Q ^{(1 )}\subsetneq Q ^{(2)}\subset R\subset \Omega _{k+m}$, so that we have violated \eqref{e.Whit}.  
Thus, we can write 
 \begin{equation*}
 Q ^{ (1 )}\cap \Omega _{k+m} = \bigcup _{R\in \mathcal R _{k} (Q)} R \,. 
\end{equation*}

In addition, for 
$ R\in  \mathcal R_k (Q)$, we must have that $ R^{(1)}\subset \Omega _{k+m}$, 
by the Whitney condition \eqref{e.Whit}.
Hence $ R^{(1)}\cap E_k(Q) = \emptyset $. See the definition of $ E_k(Q)$ in \eqref{e.EQ}. 
It follows that we have 
\begin{equation} \label{e.constant}
\mathbf 1_{R} (x)
\operatorname T_{Q^{(1)}} ^{\textup{in}} ( \mathbf 1_{E_k(Q)} \omega ) (x) 
= \mathbf 1_{R} (x)
\sum _{ \substack{ P\in \mathcal Q\\  R ^{(1 )} \subsetneqq P \subset Q ^{(1 )} }} 
\tau _P  \cdot \mathbb E _{P} ( \mathbf 1_{E_k(Q)} \omega ) \,. 
\end{equation}
In particular, the right hand side is independent of $ x\in  R $.  
Putting these observations together, we see that 
\begin{align}
\beta_k(Q) &= 
\sum _{R\in \mathcal R_k (Q)} 
\int _{R} 
f \cdot \operatorname T_{Q^{(1)}} ^{\textup{in}}  (\mathbf 1_{E_k(Q)} \omega ) \; \sigma
\\ \label{e.tau1}
& = 
\sum _{R\in \mathcal R_k (Q)} 
\int _{R}\operatorname T_{Q^{(1)}} ^{\textup{in}}   (\mathbf 1_{E_k(Q)} \omega ) \; \sigma 
\cdot  \mathbb E _{R} ^{\sigma }f \,. 
\end{align}

The maximal function $ \operatorname M _{\sigma } f$ has appeared in the last display, 
in the guise of the average $ \mathbb E _{R} ^{\sigma }f$. 
We proceed with the construction of the so-called `principal cubes.' 
This construction consists of a subcollection $ \mathcal G \subset 
\bigcup _{\substack{k \equiv M \; \textup{mod} m \\ k\ge -N}} \mathcal Q _{k}$ satisfying these two 
properties: 
\begin{gather}\label{e.PC1}
\forall\ Q\in \bigcup _{\substack{k \equiv M \; \textup{mod} m \\ k\ge -N}} \mathcal Q _{k} 
\ \ \exists G\in \mathcal G \quad \textup{so that } \quad    Q\subset G \quad \textup{and} \quad 
\mathbb E _{Q} ^{\sigma } f  \le 2  \mathbb E _{G} ^{\sigma } f \,,  
\\ \label{e.PC2}
G, G'\in \mathcal G\,, \ G\subsetneqq  G' \ \quad \textup{implies} \quad 
2\mathbb E _{G'} ^{\sigma } f < \mathbb E _{G} ^{\sigma } f \,. 
\end{gather}
It is easy to recursively construct this collection.  
Let $ \Gamma (Q)$ 
be the minimal element of $ \mathcal G$ which contains it. (So $ \Gamma (Q)$ is the `father' of $ Q$ 
in the collection $ \mathcal G$.) It follows by construction that 
$ \mathbb E _{Q}^{\sigma} f \le 2 \mathbb E _{\Gamma (Q)}^{\sigma} f$ for all $ Q$.  
A basic property of this construction, which we rely upon below is that 
\begin{equation}\label{e.PCgeo} 
\sum _{G\in \mathcal G} \mathbf 1_{G} (x) \mathbb E _{G} ^{\sigma }f  \lesssim \operatorname M _{\sigma } f  (x)\,. 
\end{equation}
Indeed, for each fixed $ x$, the terms in the series on the left are growing at least geometrically, by 
\eqref{e.PC2}, whence the sum on the left is of the order of its largest term, proving the inequality.  
It follows from \eqref{e.Lop}, that we have 
\begin{equation}\label{e.PCmax}
\sum _{G\in \mathcal G} \sigma (G) \lvert \mathbb E _{G} ^{\sigma }f\rvert ^{p} 
\lesssim \norm f. L ^{p} (\sigma ). ^{p } \,. 
\end{equation}
Both of these facts will be used below.

\begin{remark}\label{r.pc} 
Sawyer's paper on the fractional integrals \cite{MR930072} attributes this construction to 
Muckenhoupt and Wheeden \cite{MR0447956}.  In the intervening years, very similar constructions 
have been used many times, to mention just a few references, see these papers, which frequently 
use the words `corona decomposition:'  
 David and Semmes \cites{MR1251061,MR1113517}, which discuss the use of singular integrals in the 
context of rectifability.  Consult the  corona decomposition 
in \cite{MR2179730}, and the paper \cite{MR1934198} includes several examples in the context of dyadic 
analysis.  Its use in weighted inequalities appears in 
\cites{0906.1941}.  
\end{remark}

\begin{remark}\label{r.superGeometric}  It is an important point that the inequality 
\eqref{e.PC2} cannot be reversed.  Indeed, the measure $ \sigma $ is arbitrary, whence 
the averages $ \mathbb E _{G} ^{\sigma } f $ can increase dramatically as one passes from 
larger principal cubes to smaller principal cubes.  
\end{remark}

We can now 
make a further estimate of $ \beta_k(Q)$. 
Let us set 
$ \mathcal N_k (Q)= \{Q' \in \mathcal Q_k \;:\; Q'\cap  Q ^{(1 )} \neq \emptyset \}$. 
(These are the `neighbors' of $ Q$ in the collection $ \mathcal Q_k$.) 
The basic fact, a consequence of the crowd control property \eqref{e.crowd},  is that 
\begin{equation} \label{e.L<1}
\sharp \mathcal N_k (Q) \lesssim 1 \,. 
\end{equation}

Continuing from \eqref{e.tau1},  we derive a particular consequence of the  construction of $ \mathcal G$ as a subset of the Whitney cubes.   
Namely,   for $ Q \in \mathcal Q_k$,  and $ R\in \mathcal R _{k} (Q) \subset \mathcal Q _{k+m}$, we necessarily have 
$ R\subset Q'_R$ for a unique  $ Q'_R\in \mathcal N _{k} (Q)$. 
And, we have that  $ \Gamma (R) = \Gamma (Q'_R)$  or $ R\in \mathcal G$.  
This permits us to estimate 
\begin{align}
\beta_k(Q)& \le  \sum _{R \in \mathcal R_k (Q)} 
\int _{R} \operatorname T_{Q^{(1)}} ^{\textup{in}}  (\mathbf 1_{E_k(Q)} \omega ) \; \sigma 
\cdot  \mathbb E _{R}^{\sigma}f 
\\ \label{e.tau2}  
& \le \beta_{k,4} (Q)+ \beta_{k,5} (Q) \,, 
\\  \label{e.tau4}
\beta_{k,4} (Q) & \coloneqq 
\sum _ {\substack{ R \in \mathcal R_k (Q) \\ \Gamma (R)= \Gamma (Q'_R)   }}
\int _{R}  \operatorname T_{}  (\mathbf 1_{Q} \omega ) \; \sigma 
\cdot  \mathbb E _{R}^{\sigma}f 
\\ \label{e.tau5}
\beta_{k,5} (Q) & \coloneqq 
\sum _ {\substack{{R\in \mathcal R_k (Q)}\\ R\in \mathcal G }}
\int _{R}  \operatorname T_{}  (\mathbf 1_{Q} \omega ) \; \sigma 
\cdot  \mathbb E _{R}^{\sigma}f  \,. 
\end{align}
We have replaced $  \operatorname T_{Q^{(1)}} ^{\textup{in}}  (\mathbf 1_{E_k(Q)} \omega )$
by the larger term $  \operatorname T_{} (\omega \mathbf 1_{Q})$.

 We  use the 
defining condition of $ \mathcal Q_k ^{3}$, recall \eqref{e.Q3}, which gives us 
\begin{align*}
\eta 2 ^{k} \omega (Q) & \le 2 ^{k} \omega (E_k(Q) ) \le \beta_k(Q) \,, 
\\ 
\textup{whence} \qquad 2 ^{k} &\lesssim \frac {\beta_k(Q)} {\eta \omega (Q)} \,. 
\end{align*}
Thus, our estimate of  the term in \eqref{e.3MN}, $ S _{3,M}$ is given by 
\begin{align}\label{e.S3<}
 S _{3,M} & \lesssim  \eta ^{-q} \bigl[S _{4,M}+ S _{5,M} \bigr]
 \\  \label{e.vm}
 S _{v,M} & \coloneqq 
 \sum _{\substack{  k \equiv M \mod m}} \sum _{Q\in \mathcal Q_k ^{3}} 
 \frac {\omega (E_k(Q))} {\omega (Q) ^{q}} \beta _{k,v} (Q)^{q} \qquad v=4,5\,. 
\end{align}
We estimate these last two terms separately.

\subsection{The Difficult Case, Part 2.}
 Let us fix a $ G\in \mathcal G$, that is one of the principal cubes, and define 
\begin{align*}
S_{k,4}' (Q, G)& \coloneqq 
\frac {\omega (E_k(Q))} {\omega (Q) ^{q}}
\Biggl[
\sum _ {\substack{ R \in \mathcal R_k (Q) \\ \Gamma (R)= \Gamma (Q'_R) = G}}
\int _{R}  \operatorname T_{}  (\mathbf 1_{Q} \omega ) \; \sigma 
\cdot  \mathbb E _{R}^{\sigma}f
\Biggr] ^{q} 
\\ 
&\lesssim 
(\mathbb E _{G}^{\sigma}f ) ^{q}
 {\omega (E_k(Q))}  
\Biggl[ \omega (Q) ^{-1} 
\sum _ {\substack{ R \in \mathcal R_k (Q) \\ \Gamma (R)= \Gamma (Q'_R) =G}} 
\int _{R}  \operatorname T_{}  (\mathbf 1_{Q} \omega ) \; \sigma   
\Biggr] ^{q} 
\\
&\lesssim 
(\mathbb E _{G}^{\sigma}f ) ^{q}
{\omega (E_k(Q))} 
\Bigl[ \omega (Q) ^{-1}  
\int _{Q'} 
\operatorname T_{}  (\mathbf 1_{Q} \omega ) \; \sigma 
\Bigr] ^{q} 
\\
&\lesssim 
(\mathbb E _{G}^{\sigma}f ) ^{q}
{\omega (E_k(Q))} 
\Bigl[ \omega (Q) ^{-1}  
\int _{Q} 
\operatorname T_{}  (\mathbf 1_{Q'} \sigma  ) \;\omega  
\Bigr] ^{q}    &  (\textup{duality})
\\
&\lesssim 
(\mathbb E _{G}^{\sigma}f ) ^{q}
{\omega (E_k(Q))} 
\Bigl[ \omega (Q) ^{-1}  
\int _{Q} 
\operatorname T_{}  (\mathbf 1_{G} \sigma  ) \;\omega 
\Bigr] ^{q}    & (\mathbf 1_{Q'}\le \mathbf 1_{G})
\end{align*}
In the last line, we have replaced $ \mathbf 1_{Q'}$ by the larger $ \mathbf 1_{G}$, since $ Q'\subset G$, 
as $ \Gamma (Q')=G$.

The sets $ E_k(Q)$ are themselves disjoint, so that the sum above itself arises from a  linearization of the 
maximal function $ \operatorname M _{\omega  }$.  And we can estimate, again for fixed $ G\in \mathcal G$, 
\begin{align*}
\sum _{k} \sum _{Q \in \mathcal Q_k} 
S_{k,4}' (Q, G) 
& \le (\mathbb E _{G}^{\sigma}f ) ^{q}
\sum _{k} \sum _{Q \in \mathcal Q_k} {\omega (E_k(Q))} 
\Bigl[ \omega (Q) ^{-1}  
\int _{Q} 
\operatorname T_{}  (\mathbf 1_{G} \sigma  ) \;\omega  
\Bigr] ^{q}   
\\& \lesssim (\mathbb E _{G}^{\sigma}f ) ^{q}\int \operatorname M _{\omega  } (\operatorname T_{} (\mathbf 1_{G} \sigma  ) ) ^{q}\;\omega 
\\
& \lesssim (\mathbb E _{G}^{\sigma}f ) ^{q}\int \operatorname T_{} (\mathbf 1_{G} \sigma  )  ^{q}\;\omega 
\\ 
& \lesssim   \mathfrak L _{\ast }  ^{q}  (\mathbb E _{G}^{\sigma}f ) ^{q} \sigma (G) ^{q/p} \,. 
\end{align*}
Here we have used $ L ^{q} (\omega )$ bound on $ \operatorname M _{\omega  }$,  the dual testing condition, 
and the analog of \eqref{e.L'} for $ \mathfrak L _{\ast }$, which holds since we have already established 
the weak-type Theorem.

Combining these last two estimates, observe that we have the following estimate. 
\begin{align}
S _{4,M} & = \sum _{G\in \mathcal G} \sum _{k} \sum _{Q\in \mathcal Q_k} 
S'_{k,4} (Q,G) 
\\
& \lesssim   \mathfrak L _{\ast}^{q}
\sum _{G\in \mathcal G}  (\mathbb E_G^\sigma f) ^{q} \sigma (G) ^{q/p} 
\\   \label{e.again}
&\lesssim \mathfrak L _{\ast}^{q}
\Bigl[
\sum _{G\in \mathcal G}  (\mathbb E_G^\sigma f) ^{p} \sigma (G) 
\Bigr] ^{q/p}    &  (q/p\le1)
\\
&\lesssim \mathfrak L _{\ast}^{q} 
\norm f . L ^{p} (\sigma ) . ^{q} \,. 
\end{align}
The last line follows from \eqref{e.PCmax}.  
This completes the estimate for $ S _{4,M}$.

\subsection{The Difficult Case, Part 3.}
It remains to bound $ S _{5,M}$, with $ \beta_{k,5} (Q)$ as defined in \eqref{e.tau5}. 
With an abuse of notation we are going to denote the summand in the definition of $ S _{5,M}$, see 
\eqref{e.vm}, as follows.   
\begin{align} 
\beta_{k,6} (Q) & \coloneqq 
\frac {\omega (E_k(Q))} {\omega (Q) ^{q}} 
\Biggl[
\sum _ {\substack{{R\in \mathcal R_k (Q)}\\ R\in \mathcal G }} 
\int _{R}  \operatorname T_{}  (\mathbf 1_{Q} \omega ) \; \sigma 
\cdot  \mathbb E _{R}^{\sigma}f 
\Biggr] ^{q} 
\\  \label{e.BETA}
& \lesssim  \beta_{k,7}  (Q) \cdot \beta_{k,8} (Q) \,, 
\\  \label{e.beta7}
\beta_{k,7} (Q) & \coloneqq 
\frac {\omega (E_k(Q))} {\omega (Q) ^{q}}  
\Biggl[
\sum _ {\substack{{R\in \mathcal R_k (Q)}\\ R\in \mathcal G }} 
\sigma (R) ^{-p'/p} 
\Biggl(\int _{R}  \operatorname T_{}  (\mathbf 1_{Q} \omega ) \; \sigma   \Biggr) ^{p'} 
\Biggr] ^{q/p'}\,,
\\ \label{e.beta8}
\beta_{k,8} (Q) & \coloneqq 
\Biggl[
\sum _ {\substack{{R\in \mathcal R_k (Q)}\\ R\in \mathcal G }} 
\sigma (R) \cdot ( \mathbb E _{R}^{\sigma}f  ) ^{p}
\Biggr] ^{q/p}\,. 
\end{align}
Here, we have introduced the terms $ \sigma (Q') ^{\pm 1/p}$, and used the H\"older inequality 
in the $ \ell ^{p}$--$ \ell ^{p'}$ norms.  

Our first observation that the terms $ \beta_{k,7} (Q)$ admit a uniform bound.  
On the right in \eqref{e.beta7}, we use the trivial bound $ \omega (E_k(Q))\le \omega (Q)$, and push the $ p'$ inside the
integral to place ourselves in a position where we can appeal to the dual testing condition, namely \eqref{e.L'}. 
\begin{align*}
\beta_{k,7} (Q) & \lesssim 
\frac 1 {\omega (Q) ^{q-1}}  
\Biggl[
\sum _ {\substack{{R\in \mathcal R_k (Q)}\\ R\in \mathcal G }} 
\int _{R}  \operatorname T_{}  (\mathbf 1_{Q} \omega ) ^{p'} \; \sigma   
\Biggr] ^{q/p'}
\\
& \lesssim \frac 1 {\omega (Q) ^{q-1}}   \norm \operatorname T_{} ( \mathbf 1_{Q} \omega ). L ^{p'} (\sigma ). ^{q} 
\\
& \lesssim  \mathfrak L  ^{q} \frac {\omega (Q) ^{q/q'}} {\omega (Q) ^{q-1}}
\\
& \lesssim \mathfrak L ^{q} \,. 
\end{align*}

It follows from \eqref{e.vm} and \eqref{e.BETA} that we have 
\begin{align}
S _{5,M} &\lesssim \mathfrak L ^{q}
\sum _{k} \sum _{Q\in \mathcal Q_k ^{3}} \beta_{k,8} (Q) 
\\ & 
\lesssim  \mathfrak L ^{q} \Biggl[\sum _{k} \sum _{Q\in \mathcal Q_k ^{3}} \beta_{k,8} (Q) ^{p/q} \Biggr] ^{q/p} 
& (p\le q)
\\  \label{e.S5<}
&\lesssim  \mathfrak L ^{q} \Biggl[\sum _{k} \sum _{Q\in \mathcal Q_k ^{3}} 
\sum _ {\substack{{R\in \mathcal R_k (Q)}\\ R\in \mathcal G }} 
\sigma (R) \cdot ( \mathbb E _{R}^{\sigma}f  ) ^{p}
\Biggr] ^{q/p}  \,.
\end{align}
At this point, a subtle point arises.  The cubes $ R\in \mathcal R_k (Q) \subset \mathcal Q _{k+m}$, 
but a given cube $ R$ can potentially arise in many $\mathcal  Q _{k+m}$, as we noted in \eqref{e.many}. 
A given $ R$ can potentially 
arise in the sum above many times, however  this possibility  is excluded by Proposition~\ref{p.bounded} below.   
In particular, we can continue the estimate above as follows. 
\begin{align}
\eqref{e.S5<}& \lesssim  \mathfrak L ^{q} \Biggl[\sum _{G\in \mathcal G}  \sigma (G) 
(\mathbb E _{G}^{\sigma} f) ^{p} \Biggr] ^{q/p}  \lesssim \mathfrak  L ^{q} \norm f. L ^{p} (\sigma ). ^{q} \,, 
\end{align}
with the last inequality following from \eqref{e.PCmax}.  The proof of Theorem~\ref{t.strong} 
is complete, aside from the next proposition.

\begin{proposition}\label{p.bounded}[Bounded Occurrences of $ R$]  Fix a cube $ R$, and 
for  $ 1\le \ell  \le L$   suppose that 
\begin{enumerate}
\item there is an integer $ k (\ell )$ and  $ Q _{\ell } \in \mathcal Q _{k(\ell) } ^{3}$ with $ R\in \mathcal R_{k (\ell )}  (Q)$, 
\item the pairs $ (Q _{\ell }, k(\ell ) )$ are distinct. 
\end{enumerate}
We then have that $ L \lesssim 1  $, with the implied constant depending upon $ 1 $, dimension, and $ 
\eta $, 
the small constant that enters the proof at \eqref{e.Q1}---\eqref{e.Q3}. 

\end{proposition}

\begin{proof}
There are two principal obstructions to the Lemma being true:  (1) It could occur, after a potential reordering that 
$ Q _{1} \subsetneq Q _{2} \subsetneq \cdots \subsetneq Q _{L}$.  (2) It could happen that 
$ Q_1 = \cdots = Q_L$ but the $ k _{\ell }$ are distinct.  We treat these two obstructions in turn.  

Fix $ R$. We have see that $ R\subset Q ^{(1 )} _{\ell }$ for all $ 1\le \ell \le L$, 
see the paragraph after \eqref{e.RQ}. 
Suppose we have 
$
R \subset Q ^{(1 )} _{\ell _1} \subsetneq Q ^{(1 )} _{\ell _2} 
$
with $ k(\ell _1) < k(\ell _2)$.   This would violate the Whitney condition \eqref{e.Whit}.  

Let us now consider the obstruction (1)  above, namely after a relabeling, we have 
$ k( 1 ) > k(2 ) > \cdots > k({m+2})$, and 
\begin{equation*}
R \subset Q ^{(1 )} _{k(1)} \subsetneq Q ^{(1 )} _{k(2)} \subsetneq Q ^{(1 )} _{k(3)} 
\subsetneq \cdots \subsetneq Q ^{(1 )} _{k({m+ 2})} 
\end{equation*}
 This implies that $ R\in \mathcal Q _{k(1)+m}$ 
and $ R \in \mathcal Q _{k({m+2})+m}$, so that again by the nested property, 
$ R\in \mathcal Q _{k}$ for all $ k({m+ 2})+m\le k\le k(1)+m$.  Therefore, for $ s=m+2$ 
we have  $ R, Q _{s} \in \mathcal Q _{ s}$ and $ R ^{(2)} \subset Q _{s} ^{(1)
}$.  That is, $ R$ violates the Whitney condition \eqref{e.Whit}, a contradiction.  

We conclude that there are only a bounded number of positions for the cube $ Q  _{\ell } ^{(1 )}$, 
and hence a bounded number of positions for the cubes $ Q _{\ell }$.  
Thus, after a pigeonhole argument, and relabeling, we are concerned with the obstruction (2) above. 
We can after a relabeling, add to the conditions 
(1) and (2) in the Proposition that  there is a fixed cube $ \overline Q $ 
with  $ Q _{\ell } = \overline Q$ for  $ 1\le \ell \le L'$, and have 
 $ L \lesssim L'$.  This means in particular that the $ k _{\ell }$ are distinct.

The defining condition, \eqref{e.Q3}, that $ \overline Q \in \mathcal Q _{k _{\ell }} ^{3}$ means in particular that 
we have $ \omega (E_{k _{\ell }}(\overline Q )) > \eta \omega (\overline Q)$.  But, the condition that 
the $ k _{\ell }$ be distinct means that the 
sets $ E_{k _{\ell }}(\overline Q)$ are distinct, hence $ L ' \le \eta ^{-1} $ and our proof is finished.  

\end{proof}

\end{document}